\documentclass[11pt]{article}

\usepackage{color,latexsym,amsthm,amsmath,amssymb,path,enumitem,url,bbm,subcaption,geometry,bbm}

\newcommand{\ignore}[1]{}

\newtheorem{theorem}{Theorem}[section]
\newtheorem{lemma}[theorem]{Lemma}
\newtheorem{corollary}[theorem]{Corollary}

\newtheorem{claim}[theorem]{Claim}

\newcommand{\Proof}[1]
        {
        \noindent
        \emph{Proof #1.}~
        }
\newsavebox{\smallProofsym}                     
\savebox{\smallProofsym}                        %
        {
        \begin{picture}(7,7)                    %
        \put(0,0){\framebox(6,6){}}             %
        \put(0,2){\framebox(4,4){}}             %
        \end{picture}                           %
        }                                       %
\newcommand{\smalleop}[1]
        {
        \mbox{} \hfill #1~~\usebox{\smallProofsym}\!\!\!\!\!\!\
        }

\newcommand{\parag}[1]{\vspace{2mm}

\noindent{\bf #1} }

\usepackage{graphicx,psfrag}
\usepackage[rflt]{floatflt}

\usepackage{dsfont}
\newcommand{\NN}{\ensuremath{\mathbb N}}

\newcommand{\RR}{\ensuremath{\mathbb R}}

\newcommand{\pts}{\mathcal P}

\newcommand{\curves}{\Gamma}

\newcommand{\lines}{\mathcal L}

\def\eps{{\varepsilon}}

\usepackage{mathtools} 

\usepackage[colorlinks]{hyperref}

\title{Incidences with Pfaffian Curves and Functions}
\author{Alex Balsera\thanks{Newark Academy, NJ, USA.
{\sl abalsera25@newarka.edu}.}}

\begin{document}

\date{}

\maketitle

\begin{abstract}
We introduce a new approach for studying incidences with non-algebraic curves in the plane. 
This approach is based on the concepts of Pfaffian curves and Pfaffian functions, as defined by Khovanski\u{\i}.
We derive incidence bounds for curves that are defined with exponentials, logarithms, trigonometric functions, integration, and more. 
Our bound for incidences with Pfaffian curves matches a classical bound of Pach and Sharir for incidences with algebraic curves, up to polylogarithmic factors.
\end{abstract}

\section{Introduction}

Over the past 25 years, researchers discovered a surprising number of applications for point-curve incidence bounds in $\RR^2$. 
Beyond being highly useful in combinatorial proofs, such incidence bounds are also used in number theory, harmonic analysis, theoretical computer science, and more. For a few examples, see \cite{BomBour15,Bourgain05,Demeter14,KZ19}.

Traditionally, research of incidence bounds focused on algebraic curves --- curves defined by a polynomial in the plane coordinates $x,y$. 
For example, circles are defined by equations of the form $(x-p)^2+(y-q)^2 = r^2$. 
In the past decade, the rising popularity of incidence bounds led to the study of more incidence variants. 
Model theorists now study incidences in more general models such as Distal and $o$-minimal structures.
For example, see \cite{Anderson21,BR18,CGS20}.

In the current work, we introduce a new approach for  incidence bounds with non-algebraic curves.
We study \emph{Pfaffian curves} and \emph{Pfaffian functions}, as introduced and studied by Khovanski\u{\i} \cite{Khovanskii91}.
Pfaffian curves include all algebraic curves, while also allowing some use of trigonometric functions, exponents, logarithms, inverse trigonometric functions, and more. 
For example, the curves defined by $y=\tan x$ and $y=e^{x^2-5x+3}$ are Pfaffian. 

Confusingly, curves that are defined by Pfaffian functions are not the same as Pfaffian curves.
Pfaffian functions are more general than Pfaffian curves, allowing any combination of elementary mathematical functions, and beyond. (Sines and cosines are limited to an open region of $\RR^2$ where they have a bounded number of periods.)
For example, the function $e^{\tan^{3/2}x}\cdot \log y = x^{\arctan xy}$ is Pfaffian. 
Pfaffian functions also allow some use of integration. 
For example, we may consider functions such as $y-\int_0^x \frac{dt}{\ln t}$ and $y-\int_{-\infty}^x \frac{e^t}{t} dt$.
For rigorous definitions, more examples, and intuition, see Section \ref{sec:Prelims}.

\parag{Background.} Consider a point set $\pts$ and a set of curves\footnote{For rigorous definitions of curves and other concepts, see Section \ref{sec:Prelims}.} $\curves$, both in $\RR^2$. 
An \emph{incidence} is a pair $(p,\gamma)\in \pts\times \curves$ where the point $p$ is on the curve $\gamma$. 
For example, the left part of Figure \ref{fi:FirstExample} contains four points, five curves, and 13 incidences.
The \emph{incidence graph} is a bipartite graph with vertex sets $\pts$ and $\curves$.
There is an edge between the vertex of a point $p\in \pts$ and the vertex of a curve $\gamma\in \curves$ if $p\in \gamma$. 
Figure \ref{fi:FirstExample} includes an incidence graph on the right.

\begin{figure}[ht]
\centerline{\includegraphics[width=0.5\textwidth]{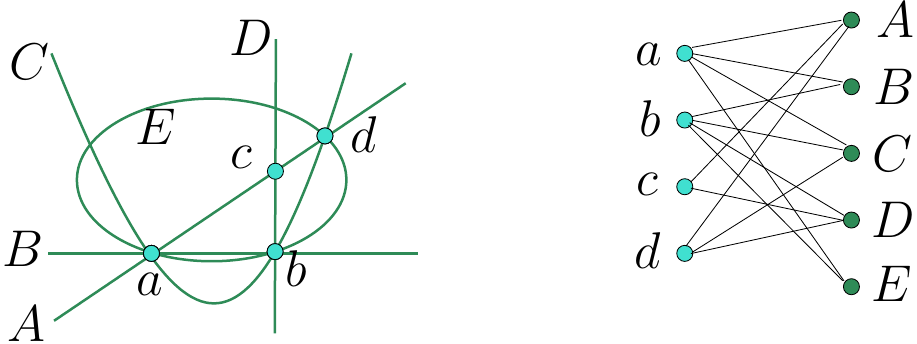}}
\caption{Left: Four points, five curves, and 13 incidences. Right: The incidence graph.}
\label{fi:FirstExample}
\end{figure}

We denote the number of incidences in $\pts\times\curves$ as $I(\pts,\curves)$.
The following theorem is a classic incidence result of Pach and Sharir \cite{PS98}. 
For a more recent variant, see Sharir and Zahl \cite{SZ17}.
The degree of an algebraic curve is the minimum degree of a polynomial that defines it.

\begin{theorem} \label{th:algebraicInc}
Consider a set $\pts$ of $m$ points and a set $\curves$ of $n$ algebraic curves of degree at most $k$, both in $\RR^2$.
If the incidence graph of $\pts$ and $\curves$ contains no copy of $K_{s,t}$ then\footnote{In $O_{k,s,t}$, the constant hidden by the $O(\cdot)$-notation may depend on $k,s,t$.}  
\[ I(\pts,\curves) = O_{k,s,t}(m^{\frac{s}{2s-1}}n^{\frac{2s-2}{2s-1}} + n + m). \]
\end{theorem}

Since two lines share at most one point, the incidence graph of points and lines contains no $K_{2,2}$.
Thus, Theorem \ref{th:algebraicInc} implies that the number of incidences between $m$ points and $n$ lines is $O(m^{2/3}n^{2/3}+m+n)$. 
The same bound holds when replacing the lines with circles of radius 1, since in that case the incidence graph contains no $K_{2,3}$.
For general circles, the number of circles that contain two given points may be arbitrarily large, but the incidence graph contain no $K_{3,2}$. 

The bound of Theorem \ref{th:algebraicInc} consists of three terms. 
The term $m$ dominates the bound when the number of points is significantly larger than the number of lines, and symmetrically for the term $n$. 
For example, in the case of lines the term $m$ dominates when the number of points is more than the square of the number of lines. 
These are considered less interesting cases.
Thus, the main term of the bound is $m^{\frac{s}{2s-1}}n^{\frac{2s-2}{2s-1}}$.

\parag{Our results.}
For incidences with Pfaffian curves, we obtain an upper bound that matches Theorem \ref{th:algebraicInc}, up to additional polylogarithmic factors.

\begin{theorem} \label{th:IncPfaffCurves}
Let $\pts$ be a set of $m$ points and let $\curves$ be a set of $n$ Pfaffian curves of degree at most $k$, both in $\RR^2$.
If the incidence graph of $\pts$ and $\curves$ contains no copy of $K_{s,t}$ then
\[ I(\pts,\curves) = O_{k,s,t}\left(m^{\frac{s}{2s-1}}n^{\frac{2s-2}{2s-1}}\log^{\frac{2s-2}{2s-1}}n + n \log^2 n + m
 \right). \] 
\end{theorem}

Recent incidence results are commonly proved by using the \emph{polynomial partitioning} technique (for example, see \cite[Chapter 3]{Sheffer22}).
So far, we were not able to use this technique with Pfaffian curves.
Instead, we rely on the older \emph{cutting} technique.
The additional polylogarimthic factors appear because we use a simplified variant of a cutting. 

Our incidence bound for curves that are defined by Pfaffian functions is weaker than the bound for Pfaffian curves.
This is not surprising, since such curves are significantly more general.  
Recall that the \emph{elementary functions} are the functions obtained by taking sums, products, roots, and compositions of finitely many polynomial, rational, trigonometric, hyperbolic, and exponential functions, including their inverses, such as arcsines and logarithms.
Similarly, \emph{Liouvillian functions} are the univariate elementary functions and their repeated integrals.
Bivariate Pfaffian functions subsume all bivariate elementary functions and all functions of the form $y=f(x)$ with a Liouvillian $f$.
Section \ref{sec:Prelims} includes the rigorous definitions and more intuition.

We define a \emph{Pfaffian family} $F$ using an open subset $U\subseteq \RR^2$ and a Pfaffian function of the form 
\[ f(x,y) = a_1 \cdot m_1(x,y) + a_2 \cdot m_2(x,y) + \cdots + a_d \cdot m_d(x,y). \]
Here, each $a_i$ is a real parameter and each $m_i(x,y)$ is a Pfaffian function that is a single term (such as a monomial).
The family $F$ consists of all curves in $U$ that are defined by $f(x,y)=0$ with some choice of values for $a_1,\ldots,a_d\in \RR$.
That is, a Pfaffian family is a set of infinitely many curves.
For example, when $U=\RR^2$ and $f(x,y) = a_1 \sqrt{e^{x-y^2}} + a_2 \tan \ln (x^y) + a_3 \int_{-\infty}^{x} \frac{e^t}{t} dt$, the family includes the curves defined by $\tan \ln (x^y) = 0$, by $100 \sqrt{e^{x-y^2}} - \tan \ln (x^y) - \frac{1}{100} \int_{-\infty}^{x} \frac{e^t}{t} dt = 0$, and by $10^{100}\sqrt{e^{x-y^2}} + 10^{-100}\tan \ln (x^y)=0$.

The \emph{dimension} of a Pfaffian family is the number of parameters $d$.
We are now ready to describe our incidence bound for curves defined by Pfaffian functions. 

\begin{theorem} \label{th:IncPfaffFuncs}
Let $F$ be a Pfaffian family of dimension $d$.
Let $\pts$ be a set of $m$ points in $\RR^2$.
Let $\curves$ be a set of $n$ curves from $F$, such that no two share a common component. 
Then, for any $\eps>0$, we have that
\[ I(\pts,\curves) = O_{d,\eps}\left(n^{\frac{2d-4}{2d-3}+\eps}m^{\frac{d-1}{2d-3}}+m+n\right). \]
\end{theorem}

The constant hidden by the $O(\cdot)$-notation in the bound of Theorem \ref{th:IncPfaffFuncs} also depends on the degree and order of the Pfaffian function that defines $F$. 
Since these values do not affect the dependency of the bound in $m$ and $n$, we do not explicitly refer to them, to keep the statement of the theorem simple. 

Pfaffian curves include all algebraic curves with no singular points, and Pfaffian functions define all algebraic curves.
Thus, every lower bound construction for algebraic curves also holds in our case. 
We can also use an algebraic configuration with many incidences to obtain many incidences with non-algebraic curves.
For example, we can start with a point--line configuration with $\Theta(m^{2/3}n^{2/3})$ incidences and apply the transformation $(x,y) \to (e^x,y)$.
This leads to a configuration of points and exponential functions with $\Theta(m^{2/3}n^{2/3})$ incidences.
So far, we did not find constructions with a larger number of incidences than currently known for algebraic curves.
(Ignoring trivial cases, such as curves that intersect infinitely many times.)
It would be interesting to know if such configurations exist.

Section \ref{sec:Prelims} is an introduction to Pfaffian curves and functions.
In Section \ref{sec:IncPfaffCurves} we prove our incidence result for Pfaffian curves.
In Section \ref{sec:IncPfaffFunc} we prove our incidence result for curves defined by Pfaffian functions.

\parag{Future work.} The current work can be seen as a modest first step of an algebraic approach to studying incidences with non-algebraic curves. 
It leads to many questions. 
A few of those:
\begin{itemize}[noitemsep,topsep=1pt]
\item Can we improve the bounds of Theorem \ref{th:algebraicInc} or Theorem \ref{th:IncPfaffFuncs}?
\item Are there constructions with Pfaffian curves or functions with more incidences than the known constructions of algebraic curves?
\item Can we revise the concept of a Pfaffian family to obtain more general incidence bounds with curves defined by Pfaffian functions? That is, can we use parameters in a more general way than for the coefficients of a fixed set of terms?
\item Are there other useful incidence models similar to Pfaffian curves and Pfaffian functions?
\end{itemize}

\parag{Acknowledgements.} We thank Moaaz AlQady for inspiring this work. 
We are grateful to Saugata Basu and Aaron Anderson for discussions about Khovanski\u{\i}'s work and incidences under more general models. 

\section{Pfaffian Curves and Pfaffian functions} \label{sec:Prelims}

In this section, we describe the basic technical definitions that are used in this paper. 
For more information, see Khovanski\u{\i} \cite{Khovanskii91}.

\parag{Pfaffian Curves.}
We rely on the topological definition of a \emph{curve}.
That is, a curve consists of a finite number of connected components, and for each component $C$ there exists a continuous function from a real interval to $C$. 

Consider a curve $\gamma\subset \RR^2$ and an open $U\subset \RR^2$, such that $\gamma\cap U$ can be parameterized as $(f^{(x)}(t),f^{(y)}(t))$.
Here, $t\in T$ for an open interval $T\subset \RR$ and $f^{(x)}(t),f^{(y)}$ are differentiable.  
By the implicit function theorem, for any curve $\gamma$ that can be expressed using a continuously differentiable function  and any point $p\in \gamma$, there exists an open set $U$ that contains $p$, such that $\gamma \cap U$ can be parameterized as above. 
For a point $p=(f^{(x)}(t_p),f^{(x)}(t_p))\in U$, the \emph{directional vector} of $\gamma$ at $p$ is $\left(\frac{d f^{(x)}}{dt}(t_p),\frac{d f^{(y)}}{dt}(t_p) \right)$.
For example, let $\gamma$ be the parabola defined by $y=x^2$ and let $U=\RR$.
We may set $f_x=t$ and $f_y=t^2$, with $T=\RR$. 
The directional vector at a point $(f^{(x)}(t_p),f^{(y)}(t_p))$ is $(1, 2t)$.

Consider an open $U\subset \RR^2$.
A \emph{vector field} $V:U \to \RR^2$ assigns a vector to each point in $U$.
For a point $p=(p_x,p_y)\in \RR^2$, we write $V(p) = (V_x(p),V_y(p))$.
For example, we may consider the vector field $V(p) = (1,2x)$.
Note that, with respect to the parabola from the preceding paragraph, the direction vector at any $p\in \gamma$ is $V(p)$. 
See Figure \ref{fi:ParabVF}(a).

 \begin{figure}[ht]
    \centering
    \begin{subfigure}[b]{0.23\textwidth}
    \centering
        \includegraphics[width=0.97\textwidth]{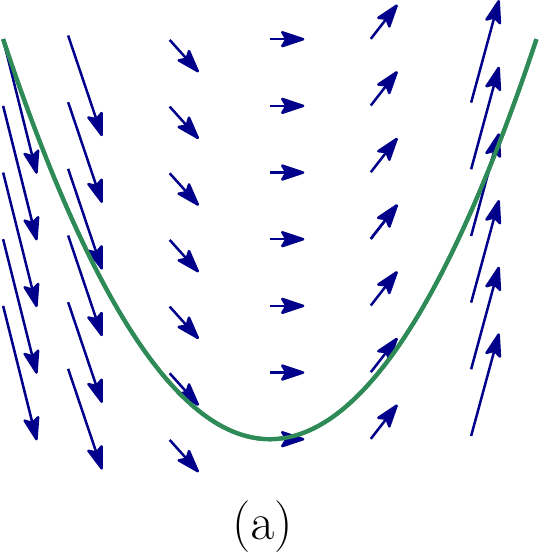}
    \end{subfigure}
    \hspace{1cm}
    \begin{subfigure}[b]{0.33\textwidth}
        \centering
        \includegraphics[width=\textwidth]{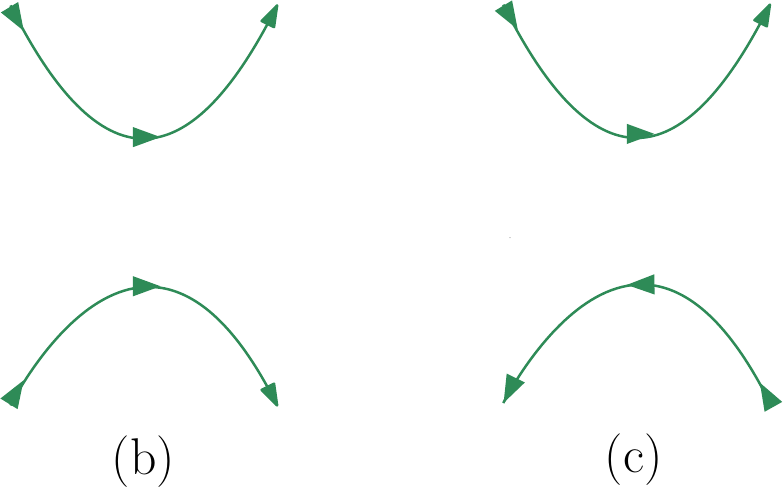}
    \end{subfigure}
    \caption{(a) The parabola $y=x^2$ and the corresponding vector field $V(p) = (1,2x)$. (b) An example that does not satisfy the third condition of a separating solution. When traveling along the upper component, the middle cell is to our right. When traveling along the bottom component, the middle cell is to our left. (c) The issue from part (b) is resolved. }
    \label{fi:ParabVF}
\end{figure}

A vector field $V$ is \emph{polynomial} if both $V_x(p)$ and $V_y(p)$ are polynomials in the coordinates of $p$.
The degree of a polynomial vector field is the larger of the degrees of $V_x(p)$ and $V_y(p)$.

Consider an open $U\subset \RR^2$ and a vector field $V:U \to \RR^2$.
A curve $\gamma\subset U$ is a \emph{separating solution} of $V$ if it satisfies:

\begin{enumerate}[noitemsep,topsep=1pt,label=(\roman*)]
\item There exists a parameterization of $\gamma$ such that, at every point $p\in \gamma$, the directional vector of $\gamma$ at $p$ is $V(p)$.
\item Every $p\in \gamma$ satisfies that $V(p)\neq (0,0)$.
\item For every connected component $\gamma_i$ of $\gamma$, the set $U \setminus \gamma_i$ is not connected. 
For each connected component $C$ of $U \setminus \gamma$, when traveling along $\gamma$ in the direction of $V$, either $C$ is always to the right, or always to the left. See Figure \ref{fi:ParabVF}(b) and Figure \ref{fi:ParabVF}(c). 
\end{enumerate}

We say that a curve $\gamma$ is \emph{Pfaffian} if it is a separating solution of a polynomial vector field.
The \emph{Pfaffian degree} of $\gamma$ is the minimum degree of a polynomial vector field that corresponds to $\gamma$. 
For an algebraic curve, the Pfaffian degree may not equal the standard algebraic degree.
By definition, the directional vector of $\gamma$ also consists of polynomials. 

Every algebraic curve with no singular points is a Pfaffian curve.
An open continuous subset of an algebraic curve with no singular points is a Pfaffian curve in an appropriate open subset of $\RR^2$.
A few other examples of Pfaffian curves:
\begin{itemize}[noitemsep,topsep=1pt]
\item The curve $y=\tan x$ corresponds to the parameterization $(t,\tan t)$ and vector field $(1,1+y^2)$. 
Indeed, the derivative of $\tan t$ is $\sec^2 t = 1+\tan^2 t$.
The open interval $T$ depends on the period of $y=\tan x$ that we are interested in.
\item The curve $y=\arctan x$ corresponds to the parameterization $(\tan t, t)$ and vector field $(1+x^2,1)$.
\item The curve $y=e^x$ corresponds to the parameterization $(t, e^t)$ and vector field $(1,y)$.
\item The curve $y=\ln(x)$ corresponds to the parameterization $(e^t, t)$ and vector field $(x,1)$.
In this case, $T=(0,\infty)$.
\item The curve $y=1/x$ corresponds to the parameterization $(t, 1/t)$ and vector field $(1,-y^2)$.
In this case, $T=(0,\infty)$ or $T=(-\infty,0)$.
\item The curve $y=1/x^{1/k}$ with $k\in \NN$ is Pfaffian, but it is difficult to show that with the straightforward parameterization $(t,t^{-1/k})$.
Instead, we can use the parameterization $(e^t, e^{-t/k})$ and vector field $(x,-y/k)$. 
\item All above examples are of the form $y=f(x)$, since these are the simplest cases. 
For a different example, consider the circle $(\cos t, \sin t)$ with $T=(0,2\pi)$. 
This corresponds to the vector field $(-y,x)$.
(This parameterization is missing one point of the circle. To be a separating solution, we need to remove that point from $\RR^2$.) 
\end{itemize}

The following lemma provides a few examples for building infinite families of Pfaffian curves.

\begin{lemma} \label{le:PcurveProps}
Consider a Pfaffian curve $\gamma$ that is parameterized as $(f^{(x)}(t),f^{(y)}(t))$ and with vector field $V(p) = (V_x(p),V_y(p))$. \\ 
(a) Applying an invertible linear transformation on $\gamma$ leads to a Pfaffian curve. \\
(b) Assume that $f^{(x)}(t) = t$ and that $\frac{d f^{(y)}}{dt} \in \RR[f^{(y)}]$ (the derivative can be expressed as a polynomial in $f^{(y)}$). 
Then, for any $p\in \RR[x]$, the curve defined by $y=f^{(y)} \circ p$ is Pfaffian. 
\end{lemma} 

For any polynomial $p\in \RR[t]$, Lemma \ref{le:PcurveProps}(b) implies that $y=e^{p(x)}$ and $y=1/p(t)$ are Pfaffian.
These are only a few examples of methods for creating Pfaffian curves.
It is not difficult to come up with variants and generalizations of Lemma \ref{le:PcurveProps}.\footnote{For example, the condition $\frac{d f^{(y)}}{dt} \in \RR[f^{(y)}]$ in part (b) can be replaced by $\frac{d f^{(y)}}{dt} =f^{(y)}\cdot q$ with $q\in \RR[t]$.}

\begin{proof}[Proof of Lemma \ref{le:PcurveProps}.]
Consider an invertible linear transformation that is defined by the matrix
\[ A = \begin{bmatrix}a_1,a_2\\ a_3,a_4 \end{bmatrix}, \quad \text{ with } a_1,a_2,a_3,a_4 \in \RR. \]

Applying the transformation $A$ on $\gamma$ leads to the curve 
\[ A\cdot (f^{(x)}(t),f^{(y)}(t)) = \begin{bmatrix}a_1\cdot f^{(x)}(t) +a_2 \cdot f^{(y)}(t)\\ a_3\cdot f^{(x)}(t)+a_4\cdot f^{(y)}(t) \end{bmatrix}. \]
The corresponding directional vector is $\big(d(a_1\cdot f^{(x)}(t) +a_2 \cdot f^{(y)}(t))/dt, d(a_3\cdot f^{(x)}(t)+a_4\cdot f^{(y)}(t))/dt\big)$.
Thus, $A\cdot \gamma$ has a polynomial vector field.
It is not difficult to verify that $A\cdot \gamma$ satisfies properties (ii) and (iii) of a separating solution (possibly in a smaller open set $U$).

(b) We denote the curve defined by $y=f^{(y)} \circ p$ as $\gamma'$. 
We observe that $\gamma$ is piece of a graph (intersecting every vertical line once, in some interval of $x$-coordinates). 
This implies that $\gamma'$ is also a graph, possibly in a smaller interval.
Thus, $\gamma'$ is a separating solution, possibly in a smaller open subset than the one corresponding to $\gamma$.
 
We note that
\[ \frac{df^{(y)}\circ p}{dt} = \left(\frac{df^{(y)}}{dt} \circ p\right) \cdot \frac{dp}{dt}. \]
Since $p\in \RR[t]$ and $\frac{d f^{(y)}}{dt} \in \RR[f^{(y)}]$, the above derivative is in $\RR[t,f^{(y)}\circ p]$, as required. 
\end{proof}

We now recall Bezout's theorem (for example, see \cite[Section 14.4]{Gibson98}). 

\begin{theorem}
Let $f$ and $g$ be nonzero polynomials in $R[x, y]$ of degrees $k_f$ and $k_g$, respectively. 
If $f$ and $g$ do not have common factors, then these two curves intersect in at most $k_f \cdot k_g$ points.
\end{theorem}

Khovanski\u{\i} proved an analogue of Bezout's Theorem for Pfaffian curves.
 
\begin{theorem} \label{th:BezPfaff}
Let $\gamma_1,\gamma_2\subset \RR^2$ be Pfaffian curves of Pfaffian degrees $k_1$ and $k_2$, respectively.
If $\gamma_1$ and $\gamma_2$ do not have a common component, then they intersect in at most $(k_1+k_2)(2k_1+k_2)+k_1+1$ points. 
\end{theorem}

We also require the following property of Pfaffian curves.

\begin{lemma} \label{le:VertTan}
Let $\gamma$ be a Pfaffian curve of Pfaffian degree $k$ that is not a segment of a vertical line. 
Then the directional vector of $\gamma$ is vertical in $O(k^2)$ points.
\end{lemma}
\begin{proof}
Let $\gamma$ be parameterized as $(f^{(x)}(t),f^{(y)}(t))$ and with vector field $V(p) = (V_x(p),V_y(p))$.
We note that $\gamma$ has a vertical tangent at a point $p\in \RR^2$ if $V_x(p)=0$. 
By definition, $V_x$ is a polynomial of degree at most $k$. 
Since $\gamma$ is not a segment of a vertical line, $\gamma$ and the curve defined by $V_x(p)=0$ do not share common components. 
By Theorem \ref{th:BezPfaff}, these two curves intersect in $O(k^2)$ points. 
\end{proof}

Finally, we mention a main weakness of the definition of Pfaffian curves. 
When two functions $f_1(x,y)$ and $f_2(x,y)$ define Pfaffian curves, it is not always true that $f_1+f_2$ defines a Pfaffian curve.

\parag{Pfaffian functions.} 
Consider an open set $U\subset \RR^2$ and let $f_1,f_2,\ldots,f_r$ be analytic functions from $U$ to $\RR$.
We say that $f_1,f_2,\ldots,f_r$ form a \emph{Pfaffian chain} if, for every $1\le i \le r$, there exist polynomials $g_{i,x},g_{i,y}\in \RR[x,y,z_1,\ldots,z_i]$ that satisfy
\begin{align*} 
\frac{\partial f_i}{\partial x} &= g_{i,x}(x,y,f_1,f_2,\ldots,f_i), \\[2mm]
\frac{\partial f_i}{\partial y} &= g_{i,y}(x,y,f_1,f_2,\ldots,f_i).
\end{align*}
The \emph{order} of the chain is $r$.
The \emph{degree} of the chain is $\max_{1\le i\le r} \max\{\deg g_{i,x},\deg g_{i,y}\}$.

A function $f$ from $U$ to $\RR$ is \emph{Pfaffian} if there exists a Pfaffian chain $f_1,f_2,\ldots,f_r$ and a polynomial $g\in \RR[x,y,z_1,\ldots, z_r]$ such that 
\[ f(x,y) = g(x,y,f_1,\ldots,f_r). \]
The \emph{order} of the Pfaffian function $f$ is the order of the corresponding Pfaffian chain. 
The \emph{degree} of the Pfaffian function $f$ is the pair $(k,\deg g)$, where $k$ is the degree of the Pfaffian chain.

As our first example of a Pfaffian function, consider $f(x,y)=x\cdot y^5-e^{x^2+3x}$.
We set $f_1(x,y) = e^{x^2+3x}$. 
Then
\begin{align*}
\frac{\partial f_1}{\partial x} &= (2x+3)\cdot e^{x^2+3x}, \\[2mm]
\frac{\partial f_1}{\partial y} &= 0, \\[2mm]
f(x,y) &= x\cdot y^5-f_1(x,y).
\end{align*}
Setting $g_{1,x}(x,y,z_1) = (2x+3)\cdot z_1$ and $g_{1,y}(x,y,z_1)=0$, implies that $f_1$ is a Pfaffian chain of order 1 and degree 2. 
Setting $g(x,y,z_1) = x\cdot y^5-z_1$ implies that $f(x,y)$ is a Pfaffian function of order 1 and degree $(2,6)$.

Every polynomial $f \in \RR[x,y]$ is a Pfaffian function of order 0 and degree $(0,\deg f)$. 
Indeed, we can take an empty Pfaffian chain and set $g=f$.
The following claim considers the more complicated example of the curve $y=\cos x$.

\begin{claim} \label{cl:cosPfaff}
The function $f(x,y)= y-\cos x$ is Pfaffian in the open set $U=\{(x,y)\in\RR^2\ :\ -\pi< x <\pi \}$.
\end{claim}
\begin{proof}
In this case, we use a Pfaffian chain of order 2. 
We set $f_1(x,y) = \tan(x/2)$ and $f_2(x,y)= \cos^2(x/2)$. 
We have that
\begin{align*}
\frac{\partial f_1}{\partial x} &= \frac{1}{2\cos^2 x} = \frac{1}{2}\left(1+\frac{1}{\cos^2 x}-1\right) = \frac{1}{2}\left(1+\tan^2 x\right), \\[2mm]
\frac{\partial f_2}{\partial x} &= -\sin(x/2)\cos(x/2), \\[2mm]
\frac{\partial f_1}{\partial y} &= \frac{\partial f_2}{\partial y} = 0.
\end{align*}
By setting $g_{1,x}(x,y,z_1)=(1+z_1)/2$, $g_{2,x}(x,y,z_1,z_2)=-z_1z_2$, and $g_{1,y}=g_{2,y}=0$, we get that We $f_1,f_2$ is a Pfaffian chain.

We set $g(x,y,z_1,z_2)=2z_2-1$.
By a half-angle trigonometric identity, we have that $f(x,y)=g(x,y,f_1,f_2)$, so $f$ is indeed Pfaffian.
The restriction on $U$ is needed only for $f_1$ to be well-defined in $U$. 
\end{proof}

Claim \ref{cl:cosPfaff} cannot be generalized to $U=\RR^2$.
Indeed, one can create a set of $m$ points and a set of $n$ scalings of the curve $y=\cos x$, where each point is incident to each curve.
Such a configuration has $mn$ incidences, which contradicts Theorem \ref{th:IncPfaffFuncs}.

Pfaffian functions also include logarithms, exponential functions, roots, inverse trigonometric functions, and more. 
Adding, multiplying, and composing Pfaffian functions yields Pfaffian functions. 

For an example that includes integration, consider a Pfaffian function $h(x,y)=y-h_1(x)$ and let $f(x,y) = y-\int_{c}^{x}h_1(t) dt$. 
(Proving this case covers all Liouvillian functions, as defined in the introduction.)
Since $h$ is Pfaffian, there exists a Pfaffian chain $f_1,\ldots, f_r$ such that $h$ is a polynomial in $x,y,f_1,\ldots, f_r$. 
We claim that $f_1,\ldots, f_r,f$ is also a Pfaffian chain.
Indeed, the derivative of $\int_{c}^{x}h(t) dt$ by $x$ is a polynomial in $x,y,f_1,\ldots, f_r$.
The derivative of $\int_{c}^{x}h(t) dt$ by $y$ is zero.
This longer Pfaffian chain implies that $f$ is a Pfaffian function. 

As usual, we require an upper bound on the number of intersection points of two curves defined by Pfaffian functions. 
Khovanski\u{\i} \cite{Khovanskii91} contains a variety of such results.
We chose to rely on Theorem 2 of Section 4.6 of that work, since it does not require introducing additional technical concepts.
The following is a special case of that theorem.

\begin{theorem} \label{th:PfaffFuncConn}
For an open $U \subset \RR^2$, let $f_1,f_2$ be pfaffian functions of the form $U \to \RR$. 
Then the number of connected components of the point set satisfying $f_1 = f_2 = 0$ is at most a finite expression that depends only on the orders and degrees of the Pfaffian functions. 
\end{theorem}

We do not state an explicit upper bound on the number of connected components in Theorem \ref{th:PfaffFuncConn}, since it is rather complicated and requires additional technical definitions. 
This theorem implies the following. 

\begin{corollary} \label{co:PfaffFuncIntersection}
Consider two curves in $\RR^2$ that are defined by Pfaffian functions and do not share an infinite connect component. 
Then the number of intersection points between the two curves is at most a finite expression that depends only on the orders and degrees of the Pfaffian functions.
\end{corollary}

\section{Incidences with Pfaffian Curves} \label{sec:IncPfaffCurves}

In this section, we derive our incidence bound with Pfaffian curves. 
That is, we prove Theorem \ref{th:IncPfaffCurves}.
Our approach follows the classical \emph{cutting} proof method.
For example, see \cite[Chapter 4]{Matousek12}.

We first recall a weaker incidence bound that does not rely on any geometric properties (for example, see \cite[Lemma 3.4]{Sheffer22}). 

\begin{lemma} \label{le:KST}
Let $\pts$ be a set of $m$ points and let $\curves$ be a set of $n$ curves, both in $\RR^2$. 
If the incidence graph of $\pts\times \lines$ contains no $K_{s,t}$, then
\begin{align*} 
I(\pts, \curves) = &O_{s,t}(mn^{1-1/s} + n), \quad \text{ and } \\[2mm]
I(\pts, \curves) &= O_{s,t}(m^{1-1/t}n + m).
\end{align*}
\end{lemma}

For example, for any three points $p,q,r \in \RR^2$, at most one circle is incident to all three points.
Thus, when studying circles the incidence graph contains no $K_{3,2}$, so Lemma \ref{le:KST} leads to $I(\pts,\curves)=O(mn^{2/3}+n)$ and $I(\pts,\curves)=O(m^{1/2}n+m)$.

Lemma \ref{le:KST} gives a weaker bound since it does not rely on any geometric properties. 
It is based only on not having $K_{s,t}$ in the incidence graph.
For that reason, this lemma is not limited to algebraic curves.
The exact same proof also holds for Pfaffian curves (and for any other definition of curves).

Let $\curves$ be a set of Pfaffian curves in $\RR^2$.
We define a \emph{pf-cell} to be a region in $\RR^2$ that is bounded by at most two segments of curves from $\curves$ and at most two vertical line segments.
A pf-cell may also be unbounded. 
For examples of pf-cells, see Figure \ref{fi:pfCells}(a--c).
A \emph{cutting} of $\curves$ is a partition of $\RR^2$ into pf-cells.
Every point of $\RR^2$ is either in the interior of a pf-cell or on the boundary of at least two pf-cells. 
For example, see Figure \ref{fi:pfCells}(d).

 \begin{figure}[ht]
    \centering
    \begin{subfigure}[b]{0.48\textwidth}
    \centering
        \includegraphics[width=0.97\textwidth]{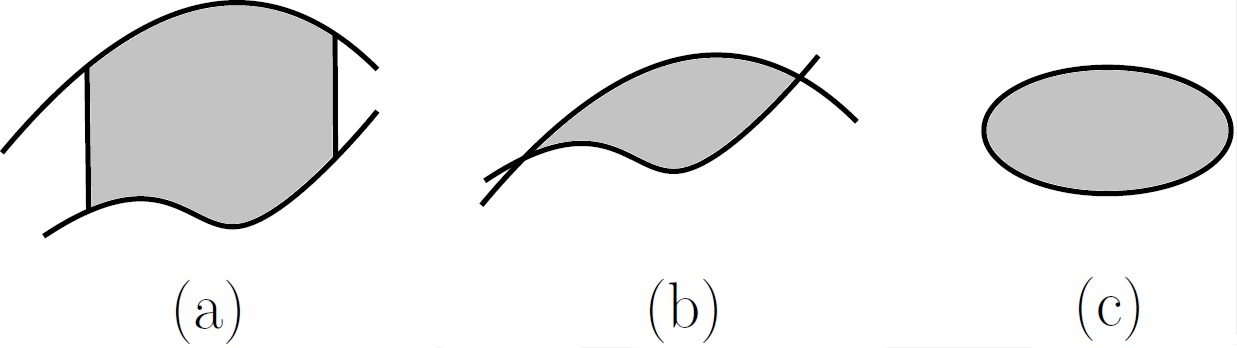}
    \end{subfigure}
    \hspace{1cm}
    \begin{subfigure}[b]{0.28\textwidth}
        \centering
        \includegraphics[width=\textwidth]{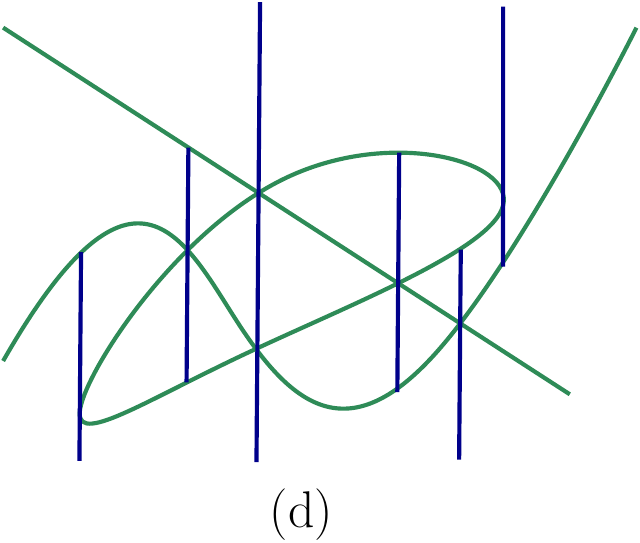}
    \end{subfigure}
    \caption{(a) A pf-cell whose boundary consists of four segments. (b) A pf-cell whose boundary consists of two segments. (c) A pf-cell whose boundary consists of one segment. (d) A cutting of the three green curves.}
    \label{fi:pfCells}
\end{figure}

\begin{lemma} \label{le:cutting}
Let $\curves$ be a set of $n$ Pfaffian curves of degree at most $k$, none of which is a segment of a vertical line. 
For each $1 \le r <n$, there exists a cutting with $O_k(r^2 \log^2 n)$ pf-cells, such that the interior of each cell is intersected by at most $n/r$ curves of $\curves$. 
\end{lemma}
\begin{proof} 
For $s\in (0,1)$ that is set below, we make $s$ random curve choices from $\curves$.
Each random choice is made uniformly, and the same curve may be chosen more than once. 
Let $S$ be the set of at most $s$ chosen curves.
A curve that is chosen more than once appears only once in $S$.

For any intersection point $p$ of two curves of $S$, we shoot vertical rays up and down from $p$. 
These rays end once they hit a curve of $S$.
If a ray never hits a curve of $S$, it continues indefinitely. 
We also shoot rays up and down from every point where the tangent of a curve of $S$ is vertical. 
For example, see Figure \ref{fi:pfCells}(d).

Let $R$ be the set of all vertical rays that were shot as described above.
By Theorem \ref{th:BezPfaff}, every two curves of $\curves$ have $O(k^2)$ intersection points.
By Lemma \ref{le:VertTan}, every curve of $S$ has a vertical tangent in $O(k^2)$ points.
Thus, we get that the number of rays in $R$ is $O_k(s^2)$.
It is not difficult to check that $S\cup R$ is a cutting.

We split every curve of $S$ into arcs at every intersection point with another curve of $S$ and at every endpoint of a ray of $R$. 
By Theorem \ref{th:BezPfaff}, every curve of $\curves$ has $O_k(s)$ intersection points, leading to a total of $O_k(s^2)$.
By Lemma \ref{le:VertTan}, there are $O_k(s)$ endpoints of rays.
That is, there are $O_k(s^2)$ arcs.
Since the boundary of each pf-cell includes 1--4 arcs and each arc is on the boundary of at most two cells, there are $O_k(s^2)$ pf-cells.

\parag{A probabilistic argument.}
To complete the proof, we now show that, with positive probability, every pf-cell is intersected by at most $n/r$ curves of $\curves$. 
This implies that there exists a choice of $S$ that satisfies this property, as required. 

We recall that the boundary of a pf-cell consists of at most two segments of curves from $\curves$, and at most two vertical rays. 
Each vertical ray originates from at most one other curve of $\curves$, so a pf-cell is defined by at most four curves. See Figure \ref{fi:FourSegments}.

\begin{figure}[ht]
\centerline{\includegraphics[width=0.23\textwidth]{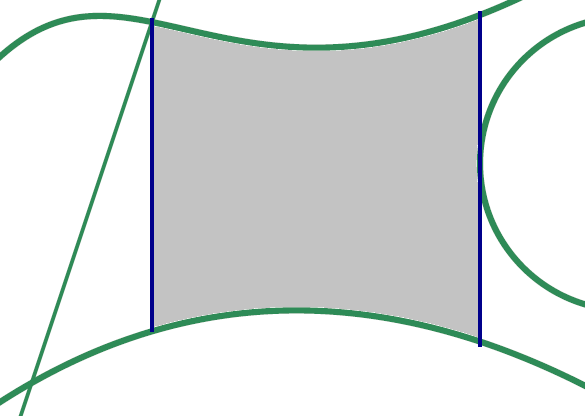}}
\caption{A pf-cell is defined by at most four curves.}
\label{fi:FourSegments}
\end{figure}

We say that a pf-cell is \emph{bad} if it is intersected by more than $n/r$ curves $\curves$. 
For a specific bad pf-cell $\tau$ to exist, the 1--4 curves of $\curves$ that define $\tau$ need to be chosen to $S$ and all curves that intersect $\tau$ need not to be chosen to $S$. 
When choosing a single curve from $\curves$, the probability of not choosing $n/r$ specific curves is $(1-1/r)$.
When performing $s$ choices, the probability of never choosing $n/r$ specific curves is $(1-1/r)^s$.
Recalling the inequality $1+x\le e^x$, we conclude that the probability of a specific bad pf-cell to exist is smaller than $e^{-s/r}$.

Since each bad cell is defined by at most four curves of $\curves$, there are fewer than $2n^4$ potential bad pf-cells.
By the union bound principle, the chance of at least one of those bad cells existing is smaller than $2n^4 e^{-s/r}$.
Setting $s=5r\log n$, we obtain that the expected number of bad cells is smaller than
\[ 2n^4 e^{-5\log n} = 2n^4 \cdot n^{-5} = 2/n. \]
As $n$ grows, this expectation becomes arbitrarily small.
Thus, with positive probability, there are no bad cells. 

To complete the proof, we recall that there are $O_k(s^2)$ pf-cells.
With $s=5r\log n$, we obtain $O(r^2\log^2 n)$ pf-cells.
\end{proof}

We are now ready to prove Theorem \ref{th:IncPfaffCurves}.
We first recall the statement of this result.
\vspace{2mm}

\noindent {\bf Theorem \ref{th:IncPfaffCurves}.}
\emph{Let $\pts$ be a set of $m$ points and let $\curves$ be a set of $n$ Pfaffian curves of degree at most $k$, both in $\RR^2$.
If the incidence graph of $\pts$ and $\curves$ contains no copy of $K_{s,t}$ then}
\[ I(\pts,\curves) = O_{k,s,t}\left(m^{\frac{s}{2s-1}}n^{\frac{2s-2}{2s-1}}\log^{\frac{2s-2}{2s-1}}n + n \log^2 n + m
 \right). \] 
\begin{proof}
We rotate $\RR^2$ so that no element of $\curves$ is a vertical lines. 
By Lemma \ref{le:PcurveProps}(a), the curves of $\curves$ remain Pfaffian after a rotation.
Similarly, a rotation does not affect $I(\pts,\curves)$ or the incidence graph.
By Lemma \ref{le:KST} and the assumption on the incidence graph, for every $\pts'\subset \pts$ and $\curves'\subset \curves$, we have that
\begin{equation} \label{eq:KSTapplication}
I(\pts',\curves') = O_{k,s,t}(|\pts'|\cdot |\curves'|^{1-1/s} + |\curves'|).
\end{equation}

By Lemma \ref{le:cutting}, for every $1 \le r <n$ there exists a cutting with $O(r^2\log^2 n)$ pf-cells, such that the interior of each pf-cell is intersected by at most $n/r$ curves of $\curves$.
We consider such a cutting with a value of $r$ that is set below. 
Let the number of pf-cells be $c=O(r^2\log^2 n)$ and denote the cells as $C_1,\ldots,C_c$.
Let $\pts_i$ be the set of points from $\pts$ in the interior of $C_i$.
Let $\curves_i$ be the set of curves from $\curves$ that intersect the interior of $C_i$.
Let $\curves_0\subset \curves$ be the set of curves that form the cutting and let $\pts_0$ be the set of points from $\pts$ that are on the curves of $\curves_0$ and the vertical rays.
That is, the curves of $\curves_0$ do not intersect the interior of any cell and the points of $\pts_0$ are not in the interior of any cell. 

\parag{Three types of incidences.}
To obtain our upper bound on $I(\pts,\curves)$, we first split this quantity into three disjoint parts:
\[ I\left(\pts,\curves \right) = \sum_{i = 1}^c I\left(\pts_i,\curves_i \right) + I\left(\pts_0,\curves\setminus\curves_0 \right) + I\left(\pts_0,\curves_0 \right). \]

\begin{figure}[ht]
\centerline{\includegraphics[width=0.23\textwidth]{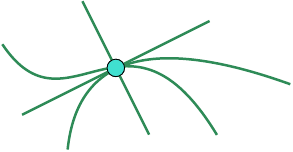}}
\caption{A point that is incident to four curves is a corner of at least eight pf-cells.}
\label{fi:Corner}
\end{figure}

We first consider $I(\pts_0,\curves_0)$.
We define a \emph{corner} of a pf-cell as a point where two of the boundary segments of that pf-cell meet.
The number of corners in a pf-cell is equivalent to the number of boundary segments (a pf-cell with one boundary segment has zero corners. See Figure \ref{fi:pfCells}(c)). 
By definition, each pf-cell has at most four corners.
If a point of $\pts_0$ is on $j$ curves of $\curves_0$, then it is a corner of at least $2j$ cells. 
For example, see Figure \ref{fi:Corner}.
Since there are $O(r^2\log^2 n)$ cells, such corners contribute $O(r^2\log^2 n)$ incidences to $I(\pts_0,\curves_0)$.
A point of $\pts_0$ that is not a corner is incident to exactly one curve of $\curves_0$. 
We conclude that
\begin{equation} \label{eq:P0C0}
I(\pts_0,\curves_0) = O(r^2\log^2 n + |\pts_0|).
\end{equation}

Next, we consider $I(\pts_0,\curves\setminus \curves_0)$.
Every point that participates in such an incidence is an intersection of a curve from $\curves\setminus \curves_0$ with a curve of and a boundary segment. 
By definition, there are $O(r^2\log^2 n)$ boundary segments and each is intersected by at most $n/r$ curves from $\curves\setminus \curves_0$.
We conclude that 
\begin{equation} \label{eq:P0C}
I(\pts_0,\curves\setminus\curves_0) = O(nr\log^2 n).
\end{equation}

It remains to derive an upper bound for $\sum_{i = 1}^c I\left(\pts_i,\curves_i \right)$.
By definition, $|\curves_i|\le n/r$ for every $i$.
Applying \eqref{eq:KSTapplication} separately in each cell leads to
\begin{align} 
\sum_{i = 1}^c I\left(\pts_i,\curves_i \right) = \sum_{i = 1}^c O_{k,s,t}(|\pts_i|\cdot |\curves_i|^{1-1/s} + |\curves_i|) &= O_{k,s,t}\left(\sum_{i = 1}^c |\pts_i|\cdot (n/r)^{1-1/s} + \sum_{i = 1}^c\frac{n}{r}\right) \nonumber \\[2mm]
&\hspace{7mm} =O_{k,s,t}\left(\frac{mn^{1-1/s}}{r^{1-1/s}} + nr\log^2n\right). \label{eq:PC} 
\end{align}

\parag{Completing the analysis.}
Combining \eqref{eq:P0C0}, \eqref{eq:P0C}, and \eqref{eq:PC} implies that
\[ I(\pts,\curves) = O_{k,s,t}\left(\frac{mn^{1-1/s}}{r^{1-1/s}} + nr\log^2n\right). \]

One term in this bound decreases with $r$ and the other increases with $r$.
Thus, to optimize the bound, we choose the value of $r$ that makes both terms equal.
This is the case when
\begin{equation} \label{eq:rVal} 
r = \frac{m^{s/(2s-1)}}{{n^{1/(2s-1)}\log^{2s/(2s-1)}}n}. 
\end{equation}
Plugging the above value of $r$ in the incidence bound leads to
\[ I(\pts,\curves) = O_{k,s,t}\left(n\cdot \frac{m^{s/(2s-1)}}{{n^{1/(2s-1)}\log^{2s/(2s-1)}}n}\cdot \log^2n\right)= O_{k,s,t}\left(m^{\frac{s}{2s-1}}n^{\frac{2s-2}{2s-1}}\log^{\frac{2s-2}{2s-1}}n\right). \]

Our choice of $r$ may violate the restriction $1\le r <n$. 
Consider the case where $r<1$. 
In this case, \eqref{eq:rVal} implies that $m<n^{1/s}\log^2 n$. 
Combining this with \eqref{eq:KSTapplication} leads to $I(\pts,\curves)= O_{k,s,t}(n\log^2 n)$.

It remains to consider the case where $r\ge n$. 
In this case, \eqref{eq:rVal} implies that $n<\sqrt{m}$.
By Theorem \ref{th:BezPfaff}, the incidence graph does not contain a $K_{6k^2+k+2,2}$.
The second part of Lemma \ref{le:KST} implies that
\[ I(\pts,\curves) = O_k(n\sqrt{m}+m) = O_k(m). \]
\end{proof}

\section{Incidences with Curves Defined by Pfaffian Functions} \label{sec:IncPfaffFunc}

In this section, we study incidences with curves defined by Pfaffian functions. 
That is, we prove Theorem \ref{th:IncPfaffFuncs}.
This requires the following incidence result from \cite{FPSSZ17}.

\begin{theorem} \label{th:HyperplanesInc} 
Let $\pts$ be a set of $m$ points and $H$ be a set of $n$ hyperplanes, both in $\RR^d$. 
For every $\eps>0$, if the incidence graph of $\pts$ and $H$ contains no $K_{s,t}$ then
\[ I(\pts,H) = O_{s,t,\eps}(m^{\frac{sd-s}{sd-1}+\eps}n^{\frac{sd-d}{sd-1}}+m+n). \]
\end{theorem}

We are now ready to prove Theorem \ref{th:IncPfaffFuncs}.
We first recall the statement of this result. 
\vspace{2mm}

\noindent {\bf Theorem \ref{th:IncPfaffFuncs}.}
\emph{Let $F$ be a Pfaffian family of dimension $d$.
Let $\pts$ be a set of $m$ points in $\RR^2$.
Let $\curves$ be a set of $n$ curves from $F$, such that no two share a common component. 
Then for any $\eps>0$, we have that}
\[ I(\pts,\curves) = O_{d,\eps}\left(n^{\frac{2d-4}{2d-3}+\eps}m^{\frac{d-1}{2d-3}}+m+n\right). \]
\begin{proof}
By Theorem \ref{th:PfaffFuncConn}, there exists $t$ such that the incidence graph of $\pts$ and $\curves$ contains no $K_{t,2}$.
The value of $t$ is upper bounded by an expression that includes the order and degree of $F$, but does not depend on $m$ and $n$. 
We move to a dual space $\RR^d$, as follows.

\parag{A dual space.} 
Let $f(x,y) = a_1 \cdot m_1(x,y) + \cdots + a_d \cdot m_d(x,y)$ be the function that defines the Pfaffian family $F$. 
Each curve $\gamma \in \curves$ corresponds to a specific choice of values for $a_1,\ldots,a_d\in \RR$.
The dual of $\gamma$ is the point $\gamma^* = (a_1,\ldots,a_d)\in \RR^d$. 
We set $\curves^* = \{\gamma^*\ :\ \gamma\in \curves\}$.
That is, $\curves^*$ is a set of $n$ points in $\RR^d$.

Let the coordinates of $\RR^d$ be $z_1,\ldots,z_d$. 
We define the dual of a point $p=(p_x,p_y)\in \RR^2$ to be the hyperplane $p^*\subset \RR^d$ that is defined by $z_1 \cdot m_1(p_x,p_y) + \cdots + z_d \cdot m_d(p_x,p_y) = 0$.
We set $\pts^* = \{p^*\ :\ p\in \pts\}$.
That is, $\pts^*$ is a set of $m$ hyperplanes in $\RR^d$.

A point $p\in \pts$ is incident to a curve $\gamma\in\curves$ if and only if the point $\gamma^*$ is incident to the hyperplane $p^*$. 
Indeed, both are equivalent to the equation $a_1 \cdot m_1(p_x,p_y) + \cdots + a_d \cdot m_d(p_x,p_y) = 0$.
This implies that $I(\pts,\curves) = I(\curves^*,\pts^*)$ and that the incidence graphs are identical except that the two sides switched.
In particular, the incidence graph of $\curves^*$ and $\pts^*$ contains no $K_{2,t}$.

We note that the origin is not in $\curves^*$.
Indeed, the origin of $\RR^d$ corresponds to $a_1=\cdots = a_d =0$, which does not define a curve in $\RR^2$. 
Assume for contradiction that there exists a line in $\RR^d$ that contains the origin and two points from $\curves^*$. 
Then there exists a nonzero $c\in \RR$ such that the points can be written as $(a_1,\ldots,a_d)$ and $(c\cdot a_1,\ldots,c\cdot a_d)$. 
This is impossible, since both points correspond to the same curve in $F$ and curves of $C$ do not share common components. 

\parag{Moving one dimension lower.}
By definition, every hyperplane of $\pts^*$ is incident to the origin of $\RR^d$.
We rotate $\RR^d$ around the origin so that no hyperplane of $\pts^*$ is defined by $z_1=0$ and no point of $\curves^*$ has a zero $z_1$-coordinate.
We note that after such a rotation the origin is still not in $\curves^*$ and no line contains the origin and two points from $\curves^*$.

Let $\Pi$ be the hyperplane in $\RR^d$ that is defined by $z_1=1$.
For a point $\gamma^*\in \curves$, let $\ell_{\gamma^*}$ be the line incident to $\gamma^*$ and the origin, and let $\gamma' = \ell_{\gamma^*}\cap \Pi$.
Intuitively, $\gamma'$ is the projection of $\gamma^*$ on $\Pi$, from the origin.
We set $\curves' = \{\gamma'\ :\ \gamma^*\in \curves^*\}$.
By the above, $\curves'$ is a set of $n$ distinct points in $\Pi$. 

For a hyperplane $p^*\in \pts^*$, we define $p' = p^* \cap \Pi$. 
This can be seen as the projection of $p^*$ on $\Pi$ from the origin.
By the above, when thinking of $\Pi$ as $\RR^{d-1}$, we get that $p'$ is a hyperplane. 
We set $\pts' = \{p'\ :\ p^*\in\pts^*\}$.
We claim that $\pts'$ is a set of $m$ distinct hyperplanes in $\RR^d$.
Indeed, given a hyperplane $p'$, we can uniquely reconstruct $p^*$ by taking the union of all lines passing through one point of $p'$ and the origin of $\RR^d$. 

It is not difficult to verify that a point $\gamma^*$ is incident to the hyperplane $p^*$ if and only if $\gamma'$ is incident to $p'$.
This implies that $I(\curves',\pts') = I(\curves^*,\pts^*)= I(\pts,\curves)$.
For the same reason, the incidence graph of $\curves'$ and $\pts'$ does not contain a $K_{2,t}$.
By thinking of $\Pi$ as $\RR^{d-1}$ and apply Theorem \ref{th:HyperplanesInc} in it, we obtain that
\begin{align*} 
I(\pts,\curves) = I(\curves',\pts') &= O_{d,\eps}\left(n^{\frac{2(d-1)-2}{2(d-1)-1}+\eps}m^{\frac{2(d-1)-(d-1)}{2(d-1)-1}}+m+n\right) \\[2mm]
&= O_{d,\eps}\left(n^{\frac{2d-4}{2d-3}+\eps}m^{\frac{d-1}{2d-3}}+m+n\right). 
\end{align*}
\end{proof}

\end{document}